\documentclass[preprint, times]{elsarticle}

\usepackage{mathtools}
\usepackage{float}
\usepackage{enumitem}
\usepackage{graphics}
\usepackage{graphicx}
\usepackage{epsfig}
\usepackage{tikz}
\usepackage[T1]{fontenc}
\usepackage{amsmath}
\usepackage{amsthm}
\usepackage{algorithm}
\usepackage{algorithmic}
\usepackage{subfig}
\usepackage{url}
\usepackage{makeidx}
\usepackage{amsfonts}
\usepackage{verbatim}
\usepackage{cite}
\usepackage{hyperref}
\usepackage{array}
\usepackage{natbib}
\hypersetup{
    colorlinks=true,       % false: boxed links; true: colored links
	allcolors=blue
}

\usepackage{color}

\newtheorem{theorem}{Theorem}

\newtheorem{lemma}{Lemma}

\newtheorem{proposition}{Proposition}

\makeatletter

\newcommand{\Rmnum}[1]{\expandafter\@slowromancap\romannumeral #1@}
\makeatother

\begin{document}

\title{Post-disaster Repair Scheduling in Partially Automated Electricity Distribution Networks}
\author[uw]{Yushi Tan}
\ead{ystan@uw.edu}
\author[ewu]{Arindam~K.~Das}
\ead{adas@ewu.edu}
\author[uw]{Payman Arabshahi}
\ead{payman@ee.washington.edu}
\author[uw]{Daniel~S.~Kirschen}
\ead{kirschen@uw.edu}

\address[uw]{Department of Electrical and Computer Engineering, University of Washington, Seattle, Washington, USA}
\address[ewu]{Department of Electrical Engineering, Eastern Washington University, Washington, USA}

\begin{abstract}
Natural disasters, such as hurricanes, large wind and ice storms, typically require the repair of a large number of components in electricity distribution networks. Since power cannot be restored before the completion of repairs, optimal scheduling of available repair crews to minimize the aggregate duration of customer interruptions reduces the harm done to the affected community. While incorporating a switch on every line of a distribution network would maximize the outage management capability and resilience, cost concerns typically prohibit such an approach. In this work, we consider the fact that the number of switches is much smaller than the number of edges in the distribution network. This generalizes the problem adopted in our previous work~\citep{tan2019scheduling}. Our modeling framework is analogous to a job scheduling problem with \emph{group soft precedence constraints} on parallel identical machines to minimize the total weighted energization time. We propose a linear programming (LP) based list scheduling algorithm and a conversion algorithm and analyze their theoretical performances.
\end{abstract}

\begin{keyword}
Disaster management, Power system resilience, Scheduling, Group soft precedence constraint
\end{keyword}

\maketitle

\section{Introduction}
% review the resilience issues and stuff
Natural disasters have caused major damage to the electricity distribution networks and deprived homes and businesses of electricity for prolonged periods. The resulting damages may also have secondary economic and environmental impact due to unpreparedness for such severe events. The recent Hurrican Harvey affected 2.02 million customers and over 6200 distribution poles were downed or damaged~\citep{nerc2018harvey}. Physical damage to grid components must be repaired before power can be restored~\citep{gridwise2013resilience,nerc2014sandy}. 
%From an operational perspective, approaches to scheduling the available repair crews to minimize the cumulative weighted customer downtime and reduce the harm done to the affected communities have been proposed in~\citep{tan2019scheduling, nurre2012restoring, coffrin2014transmission}. To the best of our knowledge, all prior work implicitly consider a fully automated distribution networks where a switch is installed for each edge. 

Distribution automation (DA) is a key component of smart grids and grid modernization. DA uses digital sensors and switches with advanced control and communication technologies to automate feeder switching; voltage and equipment health monitoring; and outage, voltage and reactive power management~\citep{doe2016da}. DA technologies, can improve distribution system resilience to extreme weather events, mostly because of the ability to isolate and locate faults~\citep{doe2016da}. Such capability is made possible by the utilization of remote controlled switches (RCS) and communication networks. 

In the absence of any isolating switch in the network, power cannot be restored to any part of the network until all repairs are completed. Presence of a switch on every edge of a distribution network, on the other hand, would allow for immediate restoration of power to the  unaffected nodes, after the damaged lines have been isolated by appropriately opening some switches. Cost concerns, however, typically prohibit such an approach. The American Recovery and Reinvestment Act (ARRA) of 2009 provided DOE with \$7.9 billion to invest in smart grid projects, with more than a quarter on the deployment of DA~\citep{doe2016da}. This led to the upgrade of 6,500 distribution circuits out of more than 200,000 in U.S.

In this work, we therefore consider that the number of switches is much smaller than the number of edges in the network. % Appropriately opening some of these switches allows the network to be \emph{islanded} in the event of disruptions and power can then be restored to portions of the network in a phased manner after completion of suitable repairs. %The topology of island $i$ is denoted by the graph $G_i = (N_i,L_i)$, where $N_i$ and $L_i$ are the sets of nodes and edges in island $i$ respectively.
Once an outage occurs, distribution system operators would open the (possibly automatic) feeder switch to avoid the danger of electrocution by loose ends downstream. 
After finishing the repairs of the damaged components and making sure there is no safety concern, the upstream switch can be closed, energizing the customers in between. Therefore, the benefit of RCS is completely undermined by the potential risk of electrical safety and the functionality RCS's is no different than that of manual switches. As a result, the impact of DA on resilience is mainly determined by the number of switches, no matter whether they are remote-controlled or manual. 
\subsection{Related Work}

This work falls within the scope of power system resilience. Approaches have been proposed for scheduling the available repair crews in order to minimize the cumulative duration of customer interruptions, and applied to general critical infrastructure~\citep{nojima1992optimal, lee2007restoration}, distribution networks~\citep{tan2019scheduling, arif2017cooptimization} and transmission networks~\citep{xu2007optimizing, nurre2012restoring, coffrin2014transmission, Nurre:2014tiba, van2015transmission}. To the best of our knowledge, all previous work implicitly consider the model of a power network where a switch is installed on each edge. Although the aforementioned work may have different formulations under various assumptions and practical concerns, they were shown to be $\mathcal{NP}$-hard in~\citep{tan2019scheduling, Nurre:2014tiba}. Many studies rely on mixed integer linear programming to solve the problem~\citep{xu2007optimizing, nurre2012restoring, lee2007restoration, coffrin2014transmission, van2015transmission} and some develop heuristics to speed up the computation~\citep{nurre2012restoring, arif2017cooptimization}. Due to computational concerns about solving such problems in real time, they are also modeled in the context of parallel machine scheduling and solved using approximation algorithms~\citep{tan2019scheduling} and heuristics~\citep{Nurre:2014tiba}.

% how the problem is different from traditional scheduling problems
Parallel machine scheduling has been studied in many applications. The problem resembles the classical parallel machine scheduling problems that considers precedence constraints and minimizes the total weighted completion time. Therefore it is necessary to clarify the differences that makes the proposed problem new.
As we will explain, the concept of energization time is the result of network flow. By taking advantage of tree topology, it relates to the completion time through Proposition~\ref{prop:EJCj}. The jobs can be scheduled in any order without any idleness, unlike the precedence constraints. But the schedules that do not follow the hierarchical relationship by tree topology might get a large energization time, or equivalently a large penalty of delay. 

% the difference between batch scheduling and scheduling groups of jobs in the previous literature

\subsection{Our Results}

We define the problem of post-disaster repairs in partially automated distribution networks in Section~\ref{sec:probDef}. After introducing the concepts of energization time and group soft precedence constraints, we model the problem by parallel machine scheduling problem with group soft precedence constraints to minimize the total weighted energization time. We present two list scheduling algorithms with performance guarantees. In Section~\ref{sec:LPlist}, the list is obtained from the results of the completion time vector linear relaxation and the algorithm has a performance guarantee of $2$. In Section~\ref{sec:swconv}, we first show that the sequencing counterpart can be solved optimally in polynomial time. The second algorithm uses the optimal sequence as the priority list and the approximation ratio is shown to be $2 - \frac{1}{m}$. 

\section{Problem Definition}
\label{sec:probDef}
A distribution network can be modeled by a tree network $G = (N,L)$, where $N$ is the set of nodes and $L$ is the set of edges. The root of $G$ is the only source node. Each node $n$ is associated with a positive weight, $w_n$, indicating its importance. Node weights can depend on multiple factors, including but not limited to, the amount of load connected to it, the type of load served, and interdependence with other critical infrastructures. 
%For example, re-energizing a node supplying a major hospital should receive a higher priority than a node supplying a similar amount of residential load. Similarly, it is conceivable that a node that provides electricity to a water sanitation plant would be assigned a higher priority. These priorities are assigned by the utilities and their determination is outside the scope of this paper. We simply assume knowledge of the $w_n$'s in this paper.
An edge in $L$ represents a distribution feeder or some other connecting component. Figure~\ref{fig:ieee123} shows a modified IEEE 123 node test feeder network, divided into 7 sub-networks (islands) by 6 switches, each illustrated using a different color.
\begin{figure}[htbp]
\centering
\subfloat[Before partitioning]{\includegraphics[width=0.45\columnwidth]{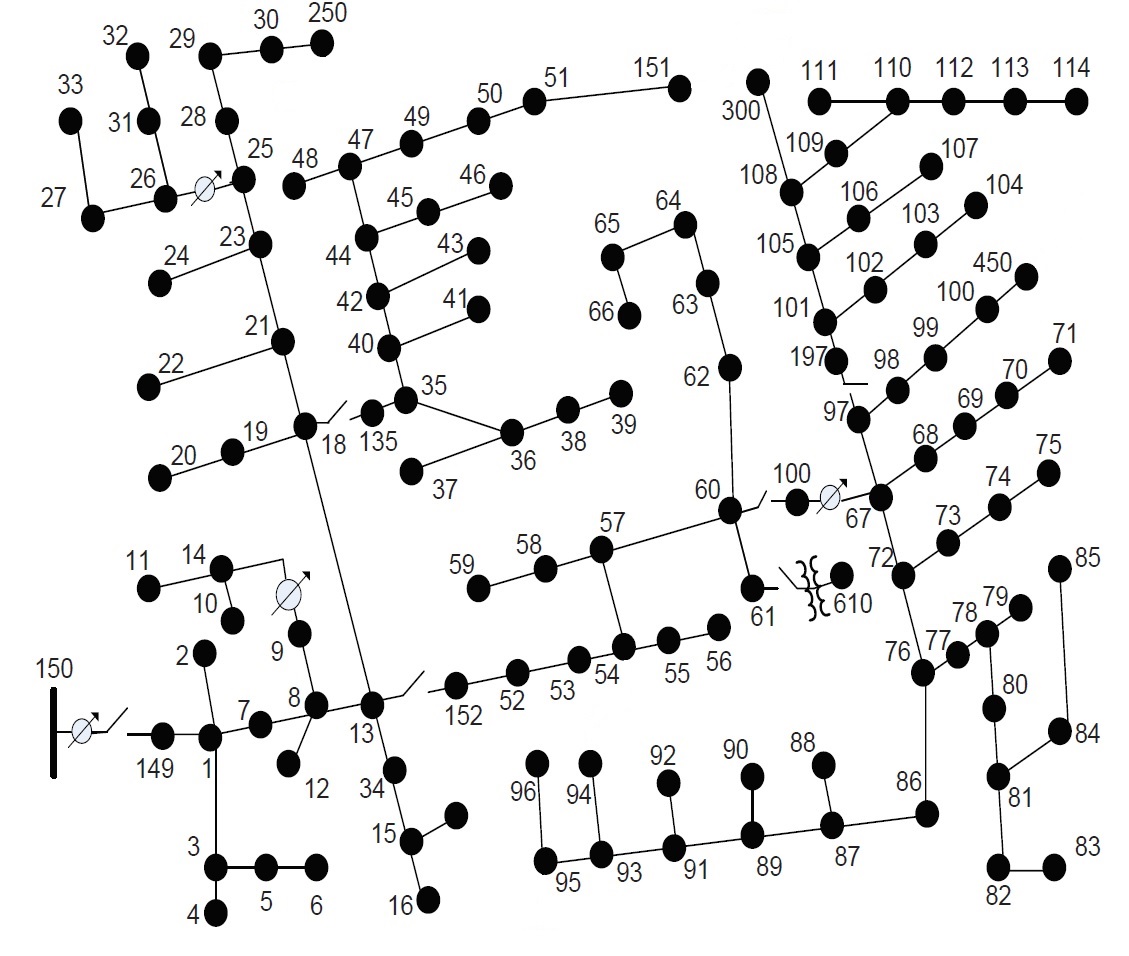}}
\subfloat[After partitioning]{\includegraphics[trim = 160 0 140 0, clip, width=0.45\columnwidth]{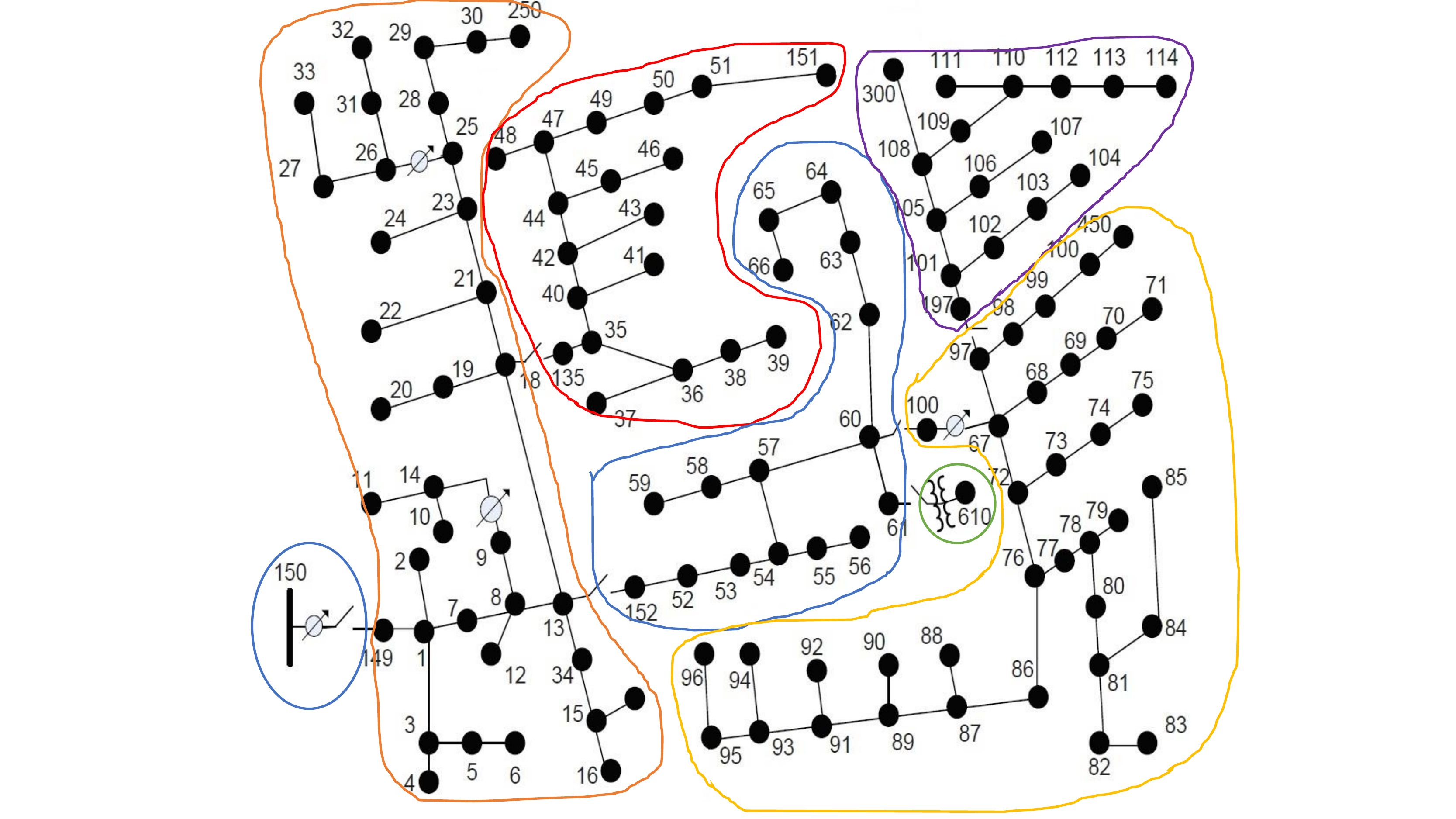}}
\caption{IEEE 123 node test feeder}
\label{fig:ieee123}
\end{figure}

Let $\mathcal{J}$ denote the set of damaged edges in $G$ and $J_i$ denote the set of damaged edges in island $i$. We will slightly abuse the notation $\mathcal{J}$ as the set of islands as well. For any disruption in island $i$, all nodes in $i$, as well as all nodes in islands downstream of $i$, will lose power due to safety concerns and lack of electrical connectivity. Let $p_j$ denote the repair time for any damaged edge $j$. We assume perfect knowledge of the set $\mathcal{J}$ and the corresponding repair times. Additionally, we assume that each job is tended to by exactly one repair team.

Instead of a rigorous power flow model, we model network connectivity using a simple network flow model, i.e., as long as a sink node is connected to the source, we assume that all loads connected to this node can be supplied without violating any voltage constraint. Without loss of generality (w.l.o.g), we assume that every edge in $G$ is damaged. If not, the undamaged edges can be assigned a repair time of $0$ so that $\mathcal{J} = L$.

%Instead of a rigorous power flow model, we model network connectivity using a simple network flow model, i.e., as long as a sink node is connected to the source, we assume that all loads connected to this node can be supplied without violating the voltage constraint. \todo{Without the necessity of modeling the voltage on each node, the nodes connected by neighboring intact will always be energized at the same time and can then be contracted into one node.} \textcolor{red}{AKD: What does the previous  sentence mean?} Therefore, without loss of generality, we assume every edge in $G$ is damaged. \textcolor{red}{AKD: I am not seeing how the penultimate sentence leads to the last sentence (usage of `therefore')' -- Two comments at the same time: it is related to the voltage. Now that we do not have to model voltage on each node, we could contract the graph. I should also add one more argument on the sufficient capacity. Change made.}

Two different time vectors are of interest: (i) a vector of completion times of line repairs, denoted by $C_j$'s, and (ii) a vector of energization times of nodes, denoted by $E_n$'s. As indicated earlier, all nodes in an island should be energized at the same time out of electrical safety concerns. Therefore, we can ascribe an energization time to \emph{island} $i$, $E_{J_i}$ and $E_{J_i} = E_n$, $\forall n \in N_i$.
\begin{figure}[h]
    \centering
    \includegraphics[width = 0.5\columnwidth]{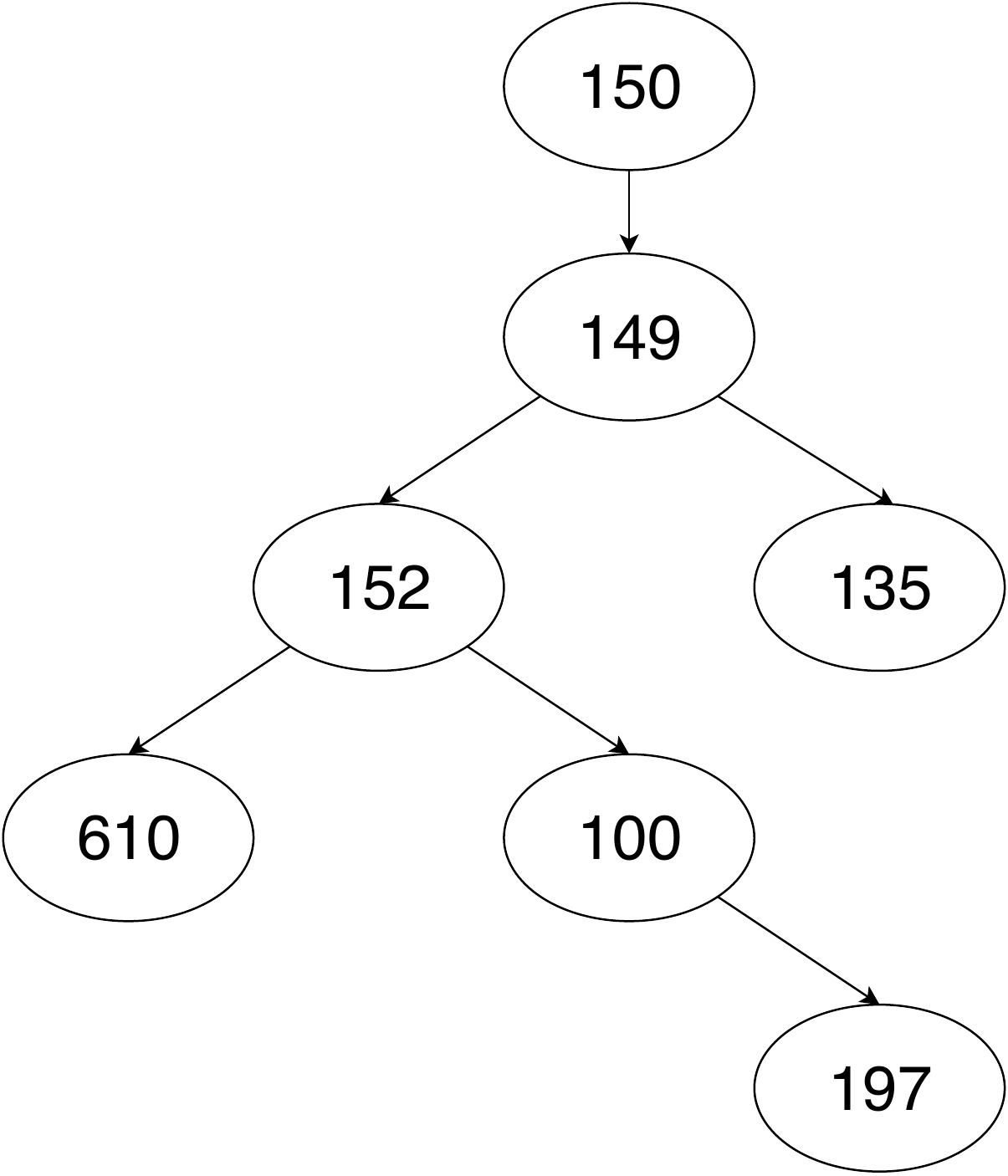}
    \caption{Group soft precedence graph, $P$, for the modified IEEE 123-node feeder shown in Fig.~\ref{fig:ieee123}(b).}
    \label{fig:grp_soft_prec_graph}
\end{figure}

Necessitated by network flow and energization constraints on islands, we now establish the notion of a \emph{group soft precedence constraint}, $J_x \prec_S J_y$, which stipulates  that island $y$ cannot be energized unless island $x$ is energized since $y$ is downstream of $x$ (in a radial topology with a single source, notions of `upstream' and `downstream' are unambiguous), or equivalently, $E_{J_y} \geq E_{J_x}$. The number of group soft precedence constraints is equal to the number of switches in the network. The corresponding group soft precedence graph, $P$, can then be obtained by further contracting non-switch lines into a supernode that represents an island. Fig.~\ref{fig:grp_soft_prec_graph} shows the group soft precedence graph for the modified IEEE 123 node test feeder in Fig.~\ref{fig:ieee123}(b). For notational simplicity, the set of soft precedence constraints represented by the precedence graph will also be denoted by $P$. These soft precedence constraints can be used to characterize the relationship between the  completion times of all jobs in some island $J$ and the energization time of island $J$, as follows.
\begin{proposition}
Let $E_J$ be the energization time of island $J$, $C_J$ the completion time of the last finished job in island $J$ (or equivalently the completion time of island $J$) and $C_j$ be the completion time of job $j$. Then:
\begin{align}
E_J \, = \underset{J^{\prime} \, \preceq_S \, J}{\max} C_{J^\prime} = \, \underset{J^{\prime} \, \preceq_S \, J}{\max} \,\; \underset{j \, \in \,  J^{\prime}}{\max} \;\, C_j
\end{align}
\label{prop:EJCj}
\end{proposition}
%

%\textcolor{red}{AKD: This para feels abrupt. Probably move it to Section 2, between eqns. (2) and (3a). Will edit later. I also suggest moving the Proposition above to the beginning of Section 2.} 
As discussed above, the node weights can be converted to an equivalent set of weights on the edges. We can therefore define the weight of an island $J$ as follows: $\omega_J := \sum_{j \in J} w_j$. Assuming $m$ identical repair teams, our scheduling framework can be modeled as a parallel machine scheduling problem with group soft precedence constraints in order to minimize the total weighted group energization time, $H := \sum \omega_J E_J$, henceforth referred to as the cost function, which can be interpreted as the \emph{aggregate harm} to the community due to loss of electrical power. A discussion on how this cost function relates to a measure of infrastructure resilience can be found in~\citep{tan2019scheduling}. 
%for discussions on how this cost function relates to the measure of infrastructure resilience.
%Similarly, the weight of node $t(l)$, $w_{t(l)}$, can be interpreted as a weight on the line $l$, $w_l$ and 

\section{LP-based list scheduling algorithm}
\label{sec:LPlist}

In this section, we discuss a scheduling algorithm derived from the LP-relaxation based on a vector of completion times $C$. For general single machine scheduling problems, a set of valid inequalities was shown in \citep{queyranne1993structure} to be the convex hull of completion time vectors. A slightly modified set of valid inequalities for general parallel machine scheduling was proposed in~\citep{schulz1996polytopes}, as stated below:
\begin{theorem}{\citep{schulz1996polytopes}}
The completion time vector $C$ of every feasible schedule on $m$ identical parallel machines satisfies the following inequalities:
\begin{align}
\sum_{j \in A} p_j C_j \geq f(A) := \frac{1}{2m}\left(\sum_{j \in A} p_j\right)^2 + \frac{1}{2} \, \sum_{j \in A} p_j^2, \;\, \forall A \in \mathcal{J}
\label{eqn:validineqparallelschulz}
\end{align}
\label{thm:parallelvalidineq}
\end{theorem}
%
%With the previously defined group soft precedence constraints and the valid inequalities for parallel machine scheduling,
\noindent We now propose the following LP-relaxation model: 
\begin{subequations}
\begin{align}
\underset{C, E}{\text{minimize}} \quad &\sum_{J \in \mathcal{J}} \omega_J E_J \\
\text{subject to} &\quad C_j \geq p_j, \; \, \forall j \in \mathcal{J} \label{eqn:swLPcnstr1}\\
&E_J \geq C_j, \; \, \forall j \in J \in \mathcal{J} \label{eqn:swLPcnstr2} \\
&E_J \geq E_{J^{\prime}}, \; \forall (J^{\prime} \rightarrow J) \in P \label{eqn:swLPcnstr3} \\
&\sum_{j \in A} p_j \, C_j \geq f(A), \; \ \forall A \subset \mathcal{J} \label{eqn:swLPcnstr4}
\end{align}
\end{subequations}
%
%\textcolor{red}{AKD: $\mathcal{J}$ is the set of all islands. Is that what you meant in the last eqn?}

Eqn.~(\ref{eqn:swLPcnstr1}) constrains the completion time of any damaged line to be lower bounded by its repair time, eqn.~(\ref{eqn:swLPcnstr2}) ensures that an island cannot be energized until all repairs in that island are completed, eqn.~(\ref{eqn:swLPcnstr3}) models the group soft precedence constraints, and eqn.~(\ref{eqn:swLPcnstr4}) describes the set of valid inequalities stated in Theorem~\ref{thm:parallelvalidineq}.
%
%\textcolor{red}{AKD: So the $M_{j}^{LP}$'s are not based on the optimal LP solution? Looks like that since you have defined $C^{LP}$ to be any feasible completion time vector, not the LP-optimal completion time vector. -- you are right. I think I need to get it corrected.}

%, i.e., it is  a solution to the constraints~(\ref{eqn:swLPcnstr1}) - (\ref{eqn:swLPcnstr4})
\begin{algorithm}
\caption{}
\label{alg:swLP_listSched}
\begin{algorithmic}
\STATE \textit{Let $C^{LP}$ denote the LP-optimal completion time vector. Define the LP mid points as $M_{j}^{LP} := C_{j}^{LP} - p_{j}/2, \ \forall j \in \mathcal{J}$. Create a job priority list by sorting $M_{j}^{LP}$'s in an ascending order. Ties are broken according to the group soft precedence constraints or arbitrarily if there is none. Whenever a crew is free, assign to it the next job from the priority list.}
\end{algorithmic}
\end{algorithm}

Next, we discuss a list scheduling algorithm based on the LP-midpoints, similar to that in~\citep{queyranne2006approximation}. Assume w.l.o.g that:
\begin{align}
    M_1^{LP} \leq M_2^{LP} \leq \cdots \leq M_{\vert \mathcal{J} \vert}^{LP}, 
    \label{eqn:Mlistorder}
\end{align}

This implies that the repair list is $(1, \dotsb, \vert L \vert)$. We will start with analyzing the basic properties of the list scheduling algorithm and the LP relaxation. Let $S^H$ and $C^{H}$ denote the vectors of starting times and completion times of jobs respectively, resulting from the list scheduling algorithm in Algorithm~\ref{alg:swLP_listSched}. 

Consider any job $j$ from that list. Since jobs $1$ to $j - 1$ are scheduled in order with no idle times in between, the start time of job $j$ satisfies:
\begin{align}
S_j^H \, &\leq \, \frac{1}{m}\sum_{i = 1}^{j - 1} p_i, 
\label{eqn:S<=sump}
\end{align}
With $M_{j}^{LP} = C_{j}^{LP} - p_{j}/2$, we rewrite eqn.~(\ref{eqn:swLPcnstr4}) as follows: %\textcolor{red}{AKD: $\mathcal{J}$ is the set of all islands. is $\forall A \subset \mathcal{J}$ correct?} 
\begin{align}
\sum_{j \in A} p_j \, M_j^{LP} \geq \frac{1}{2m}\left(\sum_{j \in A} p_j\right)^2, \;\, \forall A \subset \mathcal{J}
\label{lemma4_impeq2}
\end{align}
Now, for every $j \in \mathcal{J}$, we have:
%let $A$ denote the set of all jobs in the schedule up to and including \textcolor{red}{AKD: but not instead of and?} $j$; i.e., $A = \{1, 2, \cdots, j \}$. Then:
%Is $A$ an ordered list or just a set? -- I thought in the last CINS workshop, we agreed that () denotes an ordered list and \{\} denotes a set.
%
\begin{align}
\left(\sum_{i = 1}^j p_i\right) M_j^{LP} \ \geq \ \sum_{i = 1}^{j} p_i \, M_i^{LP} \ \geq \ \frac{1}{2m}\left(\sum_{i = 1}^{j} p_i\right)^2 \, ,
\label{lemma4_impeq3}
\end{align}
where the first inequality follows from  eqn.~(\ref{eqn:Mlistorder}) and the second inequality follows from eqn.~(\ref{lemma4_impeq2}) with $A = \{1, 2, \cdots, j \}$. Dividing both sides of eqn.~(\ref{lemma4_impeq3}) by $\sum_{i = 1}^{j} p_i$ and applying eqn.~(\ref{eqn:S<=sump}), it follows that:
\begin{align}
    S_j^{H} \leq \frac{1}{m}\sum_{i = 1}^{j - 1} p_i \leq \frac{1}{m}\sum_{i = 1}^{j} p_i \leq 2 M_j^{LP}
\end{align}
Consequently, 
\begin{align}
C_j^H = S_j^H + p_j \leq 2 M_j^{LP} + p_j = 2 \, C_j^{LP},
\end{align}
where the last equality follows from the definition of the LP-midpoints $M_{j}^{LP} := C_{j}^{LP} - p_{j}/2$. Therefore, by Proposition~\ref{prop:EJCj} and constraints~(\ref{eqn:swLPcnstr2}) and~(\ref{eqn:swLPcnstr3}), for all $J \in \mathcal{J}$, we have:
\begin{align}
E_J^H &= \underset{J^{\prime} \preceq J}{\max} \,\; \underset{j \in J^{\prime}}{\max} \;\, C_j^H \\
& \leq 2 \; \underset{J^{\prime} \preceq J}{\max} \,\; \underset{j \in J^{\prime}}{\max} \;\, C_j^{LP} \\
&= 2E_J^{LP}
\end{align}
\begin{theorem}
Algorithm~\ref{alg:swLP_listSched} is a 2-approximation.
\end{theorem}
\begin{proof}
It follows trivially from $E_J^H \leq 2 \, E_J^{LP}$ for all $J \in \mathcal{J}$ and the fact that the LP provides a lower bound on the cost of the optimal schedule.
\end{proof}
There can be exponentially many constraint~(\ref{eqn:swLPcnstr4}) in LP relaxation. The separation problem for these inequalities can be solved in polynomial time using the ellipsoid method~\citep{queyranne1993structure,schulz1996polytopes}. 
\section{A Conversion Algorithm}
\label{sec:swconv}
In this section, we first analyze the repair sequence with a single crew to derive some important insights into the general structure of the multi-crew scheduling problem. Subsequently, we convert the optimal one-crew sequence into a multi-crew schedule with a bounded performance.

\subsection{The case with one repair team, $m=1$}
First, we claim that the sequencing counterpart can be solved optimally in polynomial time by its equivalency with a known scheduling problem $1 \mid outtree \mid \sum w_{j} \, C_{j}$.
\begin{lemma}
Single crew repair scheduling in partially automated distribution networks is equivalent to $1 \mid outtree \mid \sum w_{j} \, C_{j}$, where the outtree precedences are given in the group soft precedence constraint graph $P$.
\label{lemma:swsingleequiv}
\end{lemma}
\begin{proof}
The proof heavily relies on 2 direct results of Proposition~\ref{prop:EJCj}. Consider 2 feasible schedules, $\mathbf{S}^1$ and $\mathbf{S}^2$,

\emph{Claim 1}: For a specific island $J$, if $C_{J^\prime}^{\mathbf{S}^1} \leq C_{J^\prime}^{\mathbf{S}^2}$ for all $J^\prime \preceq_s J$, then $E_{J}^{\mathbf{S}^1} \leq E_{J}^{\mathbf{S}^2}$.

\emph{Claim 2}: For a specific island $J$, if there exists a set of composite jobs $B$ such that $\max_{J^\prime \in B} C_{J^\prime}^{\mathbf{S}^1} \leq \max_{J^\prime \in A} C_{J^\prime}^{\mathbf{S}^2}$ and $C_{J^\prime}^{\mathbf{S}^1} = C_{J^\prime}^{\mathbf{S}^2}$ for all $J \in \{ J^\prime: J^\prime \preceq_s J \} - B$, then $E_{J}^{\mathbf{S}^1} \leq E_{J}^{\mathbf{S}^2}$.

We will then prove the lemma in three steps. 

\emph{Step 1}: All lines within the same island should be sequenced without interruption. To see why, suppose that the repair sequence within some island $J$, $\{ j_1, j_2, \cdots, j_n \}$, is interrupted by some other job $j^{\prime} \in J^{\prime}$ in the optimal solution; i.e., the optimal sequence contains $\{ j_1, \cdots, j_k, j^{\prime}, j_{k+1}, \cdots, j_n \}$. We now consider two cases. If $j^{\prime}$ is the last job of its own group $J^{\prime}$, then $\{ j^{\prime}, j_1, \cdots, j_n \}$ has an objective no larger than the original one. Because compared to the original sequence, $C_{J^{\prime}}$ decreases and $C_J$ along with all other completion times remains the same, and then all the energization times, according to Claim 1, will not increase. If $j^{\prime}$ is not the last job of its own group $J^{\prime}$, then $\{ j_1, \cdots, j_n, j^{\prime} \}$ has an objective no larger than the original one. Because $C_J$ decreases and $C_{J^{\prime}}$ along with all others does not change. Similarly, all the energization times will not increase. 
%\textcolor{red}{AKD: I will need help understanding this. In the last 2 sentences, $j^{\prime}$ comes at the beginning or end. In the sentence before, $j^{\prime}$ appears in the middle. I also think yu need to consider two cases, when $J^{\prime}$ is an ancestor of $J$ and when it is a successor of $J$. I do not see how $j^{\prime}$ being the last job or not of its own group is important. --because it will affect the completion time of its own group. what about the current version?}

\emph{Step 2}: Interchanging the positions of jobs within the same island does not alter the cost function. This is obvious since all nodes in the island are energized only after the last line in that island is repaired.

Steps 1 and 2 indicate that the problem is reduced to a single-crew problem with \emph{composite jobs}, each one representing an island. The processing time of each composite job is the sum of the individual repair times of jobs within the island and its weight is equal to the sum of the individual weights; i.e.,
\begin{align}
p_J = \sum_{j \, \in \, J} p_j; \,\; \  \omega_J := \sum_{j \, \in \, J} w_j.
\end{align}

\emph{Step 3}: Having shown that the problem reduces to scheduling of composite jobs to minimize $\sum \omega_J \, E_J$, we claim that the optimal sequence of composite jobs must follow outtree precedence constraints represented by the group soft precedence graph $P$.

%\textcolor{red}{AKD: NOT YET COMPLETE. TO BE COMPLETED AFTER PARTIAL MODIFICATION.}
We prove our claim by contradiction. Assume that there exists an optimal sequence, $\pi$, that does not (completely) follow the outtree precedence constraints in $P$. Let $J_y$ be the highest composite job in $P$ that is sequenced after some of its successors in $P$ and let $J_x$ be one of those successors that appears first in $\pi$. The optimal sequence, $\pi$, therefore takes the form $( \mathbf{S}_1, J_x, \mathbf{S}_2, J_y, \mathbf{S}_3 )$, where $\mathbf{S}_1$, $\mathbf{S}_2$ and $\mathbf{S}_3$ are three possibly empty sub-sequences.
%\textcolor{red}{AKD: add `(possible empty)' after sub-sequences -- why arbitrary?}
We will construct another sequence $\pi^\prime$ by altering the positions within $( J_x, \mathbf{S}_2, J_y )$. We form 2 subsequences of $\mathbf{S}_2$ corresponding to 2 kinds of jobs while retaining the original relative order from $\mathbf{S}_2$ in each of the new subsequences, %\textcolor{red}{AKD: I do not see how the original relative order in $\mathbf{S}_2$ is retained below -- in each of the new subseq.}
\begin{itemize}
    \item[$\mathbf{S}_{2-1}$]: the jobs whose energization times depend on $J_y$. This represents the jobs that are successors of $J_y$ including the successors of $J_x$.
    \item[$\mathbf{S}_{2-2}$]: the jobs whose energization times do not depend on either $J_x$ or $J_y$. This represents the jobs that are either predecessors of $J_y$ or on other branches.
\end{itemize}
Then $\pi^\prime$ takes the form of $(\mathbf{S}_1, \mathbf{S}_{2-2}, J_y, \mathbf{S}_{2-1}, J_x, \mathbf{S}_3)$. %\textcolor{red}{AKD: this is not `altering the positions within $( J_x, \mathbf{S}_2, J_y )$', as mentioned above -- what about now..} 
Now we start to analyze how the energization time changes for jobs in each subsequence from $\pi$ to $\pi^\prime$.
\begin{itemize}
    \item[$\mathbf{S}_3$]: The completion time of each composite job in $\mathbf{S}_3$ will not be affected by any permutation prior to it so neither will their energization time.
    \item[$J_x$]: In both $\pi$ and $\pi^\prime$, its energization time takes the maximum of the completion times of $J_x$, $J_y$ and potentially some of its predecessors in $\mathbf{S}_3$. Apply Claim 2 with $B = \{J_x \} \cup \{ J_y \}$. More accurately, the maximum of the completion times of $J_y$ and $J_x$ remains the same and the completion times in $\mathbf{S}_3$ do not change. The energization time of $J_x$ will not change.
    \item[$\mathbf{S}_{2-1}$]: The energization time of any composite job $J \in \mathbf{S}_{2-1}$ depends on the completion time of 1) $J$, $J_y$, its potential predecessors in $\mathbf{S}_2 \cup \{ J_x \}$ and 2) its potential predecessors in $\mathbf{S}_3$. By Claim 2, take $B = \mathbf{S}_2 \cup \{ J_x \} \cup \{ J_y \} \cup \{ J \}$ and then the energization time of composite job $J$ in $\mathbf{S}_{2-1}$ will not increase. 
    \item[$J_y$]: Apply Claim 2 with $B = \mathbf{S}_{2-2} \cup \{ J_y\}$ and its energization will not increase.
    \item[$\mathbf{S}_{2-2}$]: the completion time of each composite job in $\mathbf{S}_{2-2}$ decreases. The energization time is not dependent on $J_x$ or $J_y$ and therefore will not increase.
    \item[$\mathbf{S}_1$]: by construction of $J_x$, there is no successor of $J_y$ in $\mathbf{S}_1$ and therefore no composite job will be affected by change of locations of $J_y$ or $J_x$ or $\mathbf{S}_{2-1}$.
\end{itemize}
Combining all the facts above, $\pi^\prime$ has a smaller cost. This contradicts the fact of $\pi$ being optimal. Consequently, the optimal sequence must follow these constraints.
\end{proof}

\subsection{A conversion algorithm and its performance bound}

An algorithm for converting the optimal single crew sequence to a multi-crew schedule is given in Algorithm~\ref{alg:swLPconvert_1_to_m}. %The algorithm is still a list scheduling algorithm with a priority list of the optimal sequence. 
\begin{algorithm}
\caption{Algorithm for converting the optimal single crew sequence to an $m$-crew schedule}
\label{alg:swLPconvert_1_to_m}
\begin{algorithmic}
\STATE \textit{Treat the optimal single crew repair sequence as a priority list, and, whenever a crew is free, assign to it the next job from the list. The first $m$ jobs from the single crew repair sequence are assigned arbitrarily to the $m$ crews.}
\end{algorithmic}
\end{algorithm}
We now prove that it has an approximation bound of  $\left(2-1/m\right)$. We begin with two lemmas that provide lower bounds on the optimal cost for an $m$-crew ($2 \leq m < \infty$) schedule, in terms of the optimal costs for a single crew schedule and an $\infty$-crew schedule. Let $H^{1,\ast}$, $H^{m,\ast}$ and $H^{\infty,\ast}$ denote the optimal costs when the number of repair crews are $1$, some arbitrary $m$ ($2 \leq m < \infty$), and $\infty$ respectively. 
\begin{proposition}
$H^{m, *} \geq \frac{1}{m} \, H^{1, *}$
\label{prop:Hm1}
\end{proposition}
\begin{proof}
Given an arbitrary $m$-crew schedule $\mathbf{S}^m$ with cost $H^m$, we first construct a $1$-crew repair sequence, $\mathbf{S}^1$. We do so by sorting the \emph{energization times} of the  %groups of the damaged lines 
islands in $\mathbf{S}^m$ in an ascending order and assigning to $\mathbf{S}^1$ the sorted sequence of islands. Ties, if any, are broken according to group soft precedence constraints, or arbitrarily, if there is none. Within each island, the jobs can be permuted  arbitrarily in $\mathbf{S}^1$.

Let $C_J$ denote the completion time for island $J$. By construction, for any two islands $J$ and $J^{\prime}$ such that $J^{\prime} \prec J$, the completion time for $J^{\prime}$ can be no greater than the completion time for $J$ in $S^1$, i.e., $C_{J^{\prime}}^1 \leq C_J^1$. Therefore, $E_J^1 = C_J^1$.

Next, we claim that $E_J^1 \leq m E_J^m$, where $E_J^1$ and $E_J^m$ are the energization times of island $J$ in $\mathbf{S}^1$ and $\mathbf{S}^m$ respectively. 
%This is actually a subset of those that have smaller completion time than $E_i^m$, since completion times are not always greater than the corresponding energization times. That is,
To see why, we first observe that:
\begin{align}
E_J^1 = C_J^1 \, = \sum_{\{J^{\prime}: \ E_{J^{\prime}}^m \, \leq \, E_J^m\}} \, \sum_{j \,  \in \, J^{\prime}} p_j \, \leq \,  \sum_{\{J^{\prime}: \ C_{J^{\prime}}^m \, \leq \, E_J^m\}} \, \sum_{j \, \in \, J^{\prime}} p_j \, ,
\end{align}
where the second equality follows from the manner we constructed $S^1$ from $S^m$ and the inequality follows from the fact that all repair times are non-negative and $C_{J^{\prime}}^m \leq E_{J^{\prime}}^m \Rightarrow \{J^{\prime}: \ E_{J^{\prime}}^m \leq E_J^m\} \subseteq \{J^{\prime}: \ C_{J^{\prime}}^m \leq E_J^m\}$ for any $m$-crew schedule. For any island $J$, let $R_J$ denote the set of jobs in all such islands $J^{\prime}$ that $C_{J^{\prime}}^m \, \leq \, E_J^m$; i.e., $R_J = \bigcup_{ \{J^{\prime}: \ C_{J^{\prime}}^m \leq E_J^m\} }  J^\prime$. Partition $R_J$ into $m$ subsets, corresponding to the schedules of the $m$ crews in $\mathbf{S}^m$, as follows: $R_J = \bigcup_{k=1}^{m} R_J^k$, where $R_J^k$ is the  subset of jobs in $R_J$ that appear in the $k^{th}$ crew's schedule. It is obvious that the sum of the repair times in each $R_J^k$ can be no greater than $E_J^m$, i.e., $\sum_{j \, \in \, R_J^k} p_j \, \leq \, E_J^m$. Therefore: %\textcolor{red}{AKD: looks to me that the first inequality below should be an equality, i.e., $E_J^1 = \sum_{j \, \in \, R_J} p_j = ... $. -- (14) indicates this is an inequality.}
%of the lines 
% \forall J^{\prime} \text{ satisfies } 
%
\begin{align}
  E_J^1 \, \leq \, \sum_{j \, \in \, R_J} p_j \, = \, \sum_{k=1}^m \, \left(\sum_{j \, \in \, R_J^k} p_j\right) \, \leq \, m E_J^m
\end{align}
Proceeding with the optimal $m$-crew schedule $\mathbf{S}^{m,\ast}$, it is easy to see that $E_J^1 \leq m E_J^{m,\ast}$. Then, 
\begin{align}
%H^{m, *} &= \sum_{J \in \mathcal{J}} \Omega_J E_J^{m,*} \geq \sum_{J \in \mathcal{J}} \Omega_J \, \frac{1}{m} \, E_J^1 \\
%&= \frac{1}{m} \, H^1 \geq \frac{1}{m} \, H^{1, *}
H^{m, \ast} &= \, \sum_{J} \omega_J \, E_J^{m,\ast} \, \geq \, \sum_{J \in \mathcal{J}} \omega_J \, \frac{E_J^1}{m} \, = \, \frac{1}{m} \, H^1 \, \geq \, \frac{1}{m} \, H^{1, \ast} \label{eqn:lemma2LastEqn}
\end{align}
%\todo{I think by defn is not necessary, or if you think it is necessary to emphasize, let's do it in the paragraph before the lemma}
\end{proof}
\begin{proposition}
$H^{m,\ast} \geq H^{\infty,\ast}$
\label{prop:Hminf}
\end{proposition}
\begin{proof}
This is intuitive since the cost is minimal when the number of repair crews is at least equal to the number of damaged lines. Formally, in the $\infty$-crew case, every job can be assigned to one crew. Therefore, for any damaged line $j \in \mathcal{J}$, $C^{\infty}_j = p_j$ and, by Proposition~\ref{prop:EJCj}, $E^{\infty}_J = \underset{J^{\prime} \, \preceq_S \, J}{\max} \,\; \underset{j \, \in \,  J^{\prime}}{\max} \;\, C^{\infty}_j = \underset{J^{\prime} \, \preceq_S \, J}{\max} \,\; \underset{j \, \in \,  J^{\prime}}{\max} \;\ p_j$. Also, $C^{m,\ast}_j \geq p_j = C^{\infty}_j$. Consequently, $E^{m,\ast}_J = \underset{J^{\prime} \, \preceq_S \, J}{\max} \,\; \underset{j \, \in \,  J^{\prime}}{\max} \;\, C^{m,\ast}_j \geq \underset{J^{\prime} \, \preceq_S \, J}{\max} \,\; \underset{j \, \in \,  J^{\prime}}{\max} \;\, C^{\infty}_j = E^{\infty}_j$. Therefore:
\begin{align}
H^{m, \ast} &= \, \sum_{J} \omega_J \, E_J^{m,\ast} \, \geq \,  \sum_{J} \omega_J \, E^{\infty}_J \, = \, H^{\infty,\ast}
\end{align}
\end{proof}
%
%$\forall J \in \mathcal{J}$
\begin{theorem}
Let $E_J^m$ denote the energization time of any island $J$ after the conversion algorithm is applied to the optimal single crew repair schedule. Then: 
\[E_{J}^{m} \leq \frac{1}{m} \, E_{J}^{1,\ast} + \frac{m-1}{m} \, E^{\infty,\ast}_J.\]
\label{thm:ejm}
\end{theorem}
\begin{proof}
Let $S_j^m$ and $C_j^m$  denote respectively the start and energization times of some job $j \in \mathcal{J}$ in the $m$-crew repair schedule, $S^m$, obtained by applying Algorithm~\ref{alg:swLPconvert_1_to_m} to the optimal $1$-crew sequence, $S^{1,\ast}$. Also, let $\mathcal{I}_j$ denote the position of line $j$ in $S^{1,\ast}$ and define $ R_j := \{k: \mathcal{I}_k < \mathcal{I}_j\}$. Then:
\begin{align}
C_j^m &= S_j^m + p_j \\
&\leq \frac{1}{m} \, \sum_{i \, \in \, R_j} p_i + p_j \\
&= \frac{1}{m} \, \sum_{i \, \in \, R_j \, \cup \, j } p_i + \frac{m-1}{m} \, p_j \\
& = \frac{1}{m} \, C_j^{1,\ast} + \frac{m-1}{m} \, p_j,
\end{align}
where the inequality follows from eqn.~(\ref{eqn:S<=sump}). Using the above result and Proposition~\ref{prop:EJCj}:
\begin{align}
E_J^m &= \underset{J^{\prime} \, \preceq_S \, J}{\max} \,\; \underset{j \, \in \,  J^{\prime}}{\max} \;\,  C_j^m \\
&\leq \underset{J^{\prime} \, \preceq_S \, J}{\max} \,\; \underset{j \, \in \,  J^{\prime}}{\max} \;\, \frac{1}{m} \, C_j^{1,\ast} + \underset{J^{\prime} \, \preceq_S \, J}{\max} \,\; \underset{j \, \in \,  J^{\prime}}{\max} \;\, \frac{m-1}{m} \, p_j \\
&= \frac{1}{m} \left[ \underset{J^{\prime} \, \preceq_S \, J}{\max} \,\; \underset{j \, \in \,  J^{\prime}}{\max} \;\, C_j^{1,\ast} \right] + \frac{m-1}{m} \left[ \underset{J^{\prime} \, \preceq_S \, J}{\max} \,\; \underset{j \, \in \,  J^{\prime}}{\max} \;\, C_j^{\infty}\right] \\
&= \frac{1}{m} \, E_{J}^{1,\ast} + \frac{m-1}{m} \, E^{\infty}_J \label{eqn:th3LastEqn}
\end{align}
\end{proof}
\begin{theorem}
%The conversion algorithm is a $\left(2 - \frac{1}{m}\right)$-approximation.
The approximation bound of the single to multi-crew conversion algorithm in  Algorithm~\ref{alg:swLPconvert_1_to_m} is $\left(2 - \frac{1}{m}\right)$.
\end{theorem}
%.
\begin{proof}
The result follows from Proposition~\ref{prop:Hm1}, \ref{prop:Hminf} and Theorem~\ref{thm:ejm}.
\begin{alignat}{2}
H^m &= \sum_{J} \omega_J \, E_J^{m} && \\
&\leq \sum_{J} \omega_J \,  \left[\frac{1}{m} E_{J}^{1, \ast} + \frac{m-1}{m}E^{\infty}_J\right]   \\
&= \frac{1}{m} \, \sum_{J} \omega_J \, E_{J}^{1, \ast} + \frac{m-1}{m} \, \sum_{J} \omega_J \, E^{\infty}_J && \\
&= \frac{1}{m} \, H^{1, \ast} + \frac{m-1}{m} \, H^{\infty, \ast}   && \\
&\leq \frac{1}{m} \, \left(m \, H^{m, \ast}\right) + \frac{m-1}{m} \, H^{m, \ast} \\
&= \left(2 - \frac{1}{m}\right) \, H^{m, \ast} &&
\end{alignat}
\end{proof}

% \textcolor{red}{AKD: have not looked at the rest}
% \begin{remark}
% \todo{re-evaluate the technical meaning of this argument}
% Algorithm~\ref{alg:swLPconvert_1_to_m} in some sense imposes a group precedence constraint defined as follows.

% If there is a group precedence constraint $J_{k_1} \prec J_{k_2}$ for any $k_1, k_2 \in \{1, \cdots n \}$, it must satisfy that for any $j_1 \in J_{k_1}$ and $j_2 \in J_{k_2}$, $S_{j_1} \leq S_{j_2}$.

% However, such a constraint is not always satisfied in the optimal solution. For example, in Fig.~(\ref{fig:counterexgroupprec}), where $1$, $1^{\prime}$ and $1^{\prime \prime}$ belongs to a group $1$ and $2$ is a standalone job, the schedule at the bottom will outperforms if $2$ is the child of group $1$ and $\Omega_2 \gg \Omega_1$.
% \begin{figure}[htbp]
% \centering
% \includegraphics[width = \columnwidth]{counterexamplegroupprecedence.pdf}
% \caption{Case when group precedence is not optimal}
% \label{fig:counterexgroupprec}
% \end{figure}
% \end{remark}

As mentioned in proof of Lemma~\ref{lemma:swsingleequiv}, interchanging positions within the same group does not affect the scheduling objective in sequencing repairs with one repair team. And by Algorithm~\ref{alg:swLPconvert_1_to_m}, interchanging positions within the same group does not affect the performance guarantee. In fact, Algorithm~\ref{alg:swLPconvert_1_to_m} does not specify a schedule but rather a family of schedules.

Recall that Graham's \emph{List} scheduling algorithm with arbitrary priority list is a $2 - \frac{1}{m}$ approximation for makespan minimization. If there is only one island, the problem is equivalent to minimizing the makespan. Now that changing positions within the island does not affect the approximation ratio, $2 - \frac{1}{m}$ is the tightest bound possible for Algorithm~\ref{alg:swLPconvert_1_to_m}.

Although the performance bound could not be improved by finding the best within-group schedule, the objectives might be improved for some cases.

\section{Conclusion and Future Work}

In this paper, we introduced the problem of post-disaster repair scheduling in partially automated electricity distribution networks and solved it by 2 algorithms with constant factor guarantees on the worst case performances. 
%The first method, based on LP-relaxation, is proven to be a $2$-approximation algorithm. The second method converts the optimal single crew schedule, solvable in polynomial time, to an arbitrary $m$-crew schedule with a performance bound of $\left(2-\frac{1}{m}\right)$. The third method, based on $\rho$-factors which can be interpreted as component importance measures, is shown to be equivalent to the conversion algorithm. 

%Although we have focused on electricity distribution networks, the heuristic algorithms can also be applied to any infrastructure network with a radial structure (e.g., water distribution networks). 
The solution approach paves the way for planning the installation of switches in resilient distribution networks given a budget constraint. And a cost comparison with the component hardening approach~\citep{tan2018distribution} will also be very useful for utilities to make planning decisions.
Another major direction of our future work includes development of efficient algorithms which can be applied to distribution networks with both normally closed and normally open switches, allowing for network reconfiguration. This capability will allow some nodes to be connected to a substation via more than one path~\citep{baran1989network}.  

\section*{References}

\bibliography{switch}

\end{document}